\documentclass[reqno, a4paper,]{amsart}
\usepackage{amsmath,amssymb,dsfont,verbatim}

\usepackage[raggedright]{titlesec}

\titleformat{\chapter}[display]
{\normalfont\huge\bfseries}{\chaptertitlename\\thechapter}{20pt}{\Huge}
\titleformat{\section}
{\normalfont\Large\bfseries\center}{\thesection}{1em}{}
\titleformat{\subsection}
{\normalfont\large\bfseries}{\thesubsection}{1em}{}
\titleformat{\subsubsection}[runin]
{\normalfont\normalsize\bfseries}{\thesubsubsection}{1em}{}
\titleformat{\paragraph}[runin]
{\normalfont\normalsize\bfseries}{\theparagraph}{1em}{}
\titleformat{\subparagraph}[runin]
{\normalfont\normalsize\bfseries}{\thesubparagraph}{1em}{}
\titlespacing*{\chapter} {0pt}{50pt}{40pt}
\titlespacing*{\section} {0pt}{3.5ex plus 1ex minus .2ex}{2.3ex plus .2ex}
\titlespacing*{\subsection} {0pt}{3.25ex plus 1ex minus .2ex}{1.5ex plus .2ex}
\titlespacing*{\subsubsection}{0pt}{3.25ex plus 1ex minus .2ex}{1.5ex plus .2ex}
\titlespacing*{\paragraph} {0pt}{3.25ex plus 1ex minus .2ex}{1em}
\titlespacing*{\subparagraph} {\parindent}{3.25ex plus 1ex minus .2ex}{1em}
\input xypic
\xyoption{all}



\newtheorem{theorem}{Theorem}[section]
\newtheorem{lemma}[theorem]{Lemma}
\newtheorem{proposition}[theorem]{Proposition}
\newtheorem{corollary}[theorem]{Corollary}

\theoremstyle{definition}

\newtheorem{notations}[theorem]{Notations}

\theoremstyle{remark}
\newtheorem{remark}[theorem]{Remark}

\DeclareMathOperator{\Supp}{Supp}
\DeclareMathOperator{\Z}{Z}

\newcommand{\ov}{\overline}

\numberwithin{equation}{section}

\begin{document}

\title{On the centralizers in the Weyl algebra}

\author{Jorge A. Guccione}
\address{Departamento de Matem\'atica\\ Facultad de Ciencias Exactas y Naturales, Pabell\'on~1, Ciudad Universitaria\\ (1428) Buenos Aires, Argentina.}
\email{vander@dm.uba.ar}
\thanks{Supported by  UBACYT 095, PIP 112-200801-00900 (CONICET) and
PUCP-DAI-2009-0042}

\author{Juan J. Guccione}
\address{Departamento de Matem\'atica\\ Facultad de Ciencias Exactas y Naturales\\ Pabell\'on~1, Ciudad Universitaria\\ (1428) Buenos Aires, Argentina.}
\email{jjgucci@dm.uba.ar}
\thanks{Supported by  UBACYT 095 and PIP 112-200801-00900 (CONICET)}

\author{Christian Valqui}
\address{Pontificia Universidad Cat\'olica del Per\'u - Instituto de Matem\'atica y Ciencias Afi\-nes, Secci\'on Matem\'aticas, PUCP, Av. Universitaria 1801, San Miguel, Lima 32, Per\'u.}
\email{cvalqui@pucp.edu.pe}
\thanks{Supported by PUCP-DAI-2009-0042,  Lucet 90-DAI-L005, SFB 478 U.
M\"unster, Konrad Adenauer Stiftung.}

\thanks{The second author thanks the appointment as a visiting professor ``C\'atedra Jos\'e Tola Pasquel'' and the hospitality during his stay at the PUCP}

\begin{abstract} Let $P,Q$ be elements of the Weyl algebra $W$. We prove that if $[Q,P]=1$, then the centralizer of $P$ is the polynomial algebra $k[P]$.
\end{abstract}

\maketitle

\section*{Introduction} Let $k$ be a characteristic zero field. The Weyl algebra $W$ of index~$1$ over $k$ is the unital associative $k$-algebra generated by elements $X,Y$ and the relation $[Y,X]=1$. This algebra was introduced by Hermann Weyl in order to study the Heisenberg uncertainty principle in quantum mechanics. A detail analysis of $W$ was made in~\cite{D}. Among other things, in this paper the author establishes many interesting properties about the centralizer $\Z(P)$ of an element $P$. Another important paper devoted to the investigation of centralizers of elements in the Weyl algebra is~\cite{B}. In this note we continue the study of $\Z(P)$. Our main result is that if  $P, Q\in W$ satisfy $[Q,P]=1$, then $\Z(P)=k[P]$. Dixmier asked in~\cite{D} if each endomorphism of $W$ is an automorphism. An affirmative answer immediately implies our theorem, by~\cite[Th. 9.1]{D}.

\section{Preliminaries}
In this Section we establish some notations and we recall some results from~\cite{D}. Let $P$ and $Q$ be non zero elements of $W$.

\begin{notations} For $P=\sum a_{ij} X^iY^j$, we write

\begin{itemize}

\smallskip

\item[-] $v(P):= \max\{i-j: a_{ij}\ne 0\}$,

\smallskip

\item[-] $\ell(P):=\displaystyle\sum_{i-j=v(P)} a_{ij} X^iY^j$,

\smallskip

\item[-] $\Supp(P):=\{(i,j): a_{ij}\ne 0\}$,

\smallskip

\item[-] $w(P):=(i_0,i_0-v(P))$ such that $i_0=\max\{i: (i,i-v(P))\in \Supp (\ell(P))\}$,

\smallskip

\item[-] $\ell_t(P):= a_{i_0j_0}X^{i_0}Y^{j_0}$, where $(i_0,j_0)= w(P)$,

\smallskip

\item[-] $\ell_c(P):= a_{i_0j_0}$, where $(i_0,j_0)= w(P)$,

\smallskip

\item[-] $W_+:=\{ P\in W: v(P)>0 \}$,

\smallskip

\item[-] If $\ell_c(P) = 1$ then we will say that $P$ is {\em monic}.

\smallskip

\end{itemize}

\end{notations}

We say that $P$ is aligned with $Q$ and write $P\sim Q$, if
$$
km=jl\qquad \text{with $w(P)=(k,j)$ and $w(Q)=(l,m)$}
$$
Note that $\sim$ is not an equivalence relation (it is so restricted to $\{P:w(P)\ne (0,0)\}$).

\begin{proposition}\label{conmutadores} The following facts hold:

\begin{enumerate}

\smallskip

\item $X^kY^jX^lY^m=X^{k+l}Y^{j+m} + \sum_{i=1}^{\min(j,l)} i!\binom{j}{i}\binom{l}{i} X^{k+l-i} Y^{j+m-i}$.

\smallskip

\item If $P \not\sim Q$, then $[P,Q]\ne 0$ and $w\bigl([P,Q]\bigr) = w(P)+w(Q)-(1,1)$.

\smallskip

\item $[\ell(P),\ell(Q)] = 0$ implies $P\sim Q$.

\smallskip
\end{enumerate}
\end{proposition}

\begin{proof} (1)\enspace It suffices to prove it when $k=m=0$. In this case it follows easily by induction on $j$, using that
$$
[Y,X^l] = lX^{l-1}\quad\text{and}\quad [Y^j,X^l] = [Y,X^l]Y^{j-1}+Y[Y^{j-1},X^l].
$$

\smallskip

\noindent (2)\enspace Let $w(P) = (k,j)$ and $w(Q) = (l,m)$. Since $P\not\sim Q$, it follows from item~(1) that
$$
\ell_t\bigl([\ell_t(P),\ell_t(Q)]\bigr) = \Biggl(\!\!\binom{j}{1}\binom{l}{1} - \binom{m}{1} \binom{k}{1}\!\!\Biggr) \ell_c(P)\ell_c(Q)X^{k+l-1}Y^{j+m-1}.
$$
From this fact and item~(1) it follows now that
$$
\ell_t\bigl([P,Q]\bigr) = \Biggl(\!\!\binom{j}{1}\binom{l}{1} - \binom{m}{1} \binom{k}{1}\!\! \Biggr) \ell_c(P)\ell_c(Q)X^{k+l-1}Y^{j+m-1},
$$
and so $w\bigl([P,Q]\bigr) = w(P)+w(Q)-(1,1)$.

\smallskip

\noindent (3)\enspace If $P\not\sim Q$, then $\ell(P)\not\sim \ell(Q)$, since $w(\ell(P)) = w(P)$ and $w(\ell(Q)) = w(Q)$. Thus, by item~(2) we have $[\ell(P),\ell(Q)]\ne 0$.
\end{proof}

From item~(1) of Proposition~\ref{conmutadores} it follows immediately that
\begin{align*}
& v(PQ)=v(P)+v(Q),\\
& w(PQ)=w(P)+w(Q),\\
&\ell(PQ)=\ell(P)\ell(Q),\\
&\ell_t(PQ)=\ell_t(\ell_t(P)\ell_t(Q)),\\
&\ell_c(PQ)=\ell_c(P)\ell_c(Q).
\end{align*}

\begin{remark} Let $\overline{v}(P):=\max \{j-i : a_{ij}\ne 0\}$. All the results in this section admit symmetric versions with $\overline{w},\overline{\ell},\overline{\ell}_t, \overline{\ell}_c$ and $\overline{W}_+$, defined mimicking the definition of $w,\ell,\ell_t, \ell_c$ and $W_+$, using $\overline{v}$ instead of $v$. In the sequel, when we introduce any symbol denoting an object, the same symbol with an overline will denote the symmetric object.
\end{remark}

For each $j\in \mathds{Z}$ we set
$$
W_j:=\{P\in W\setminus\{0\} : P=\ell(P) \text{ and } v(P)=j\}\cup \{0\}.
$$
Clearly $W_j$ is a subvector space of $W$ and $W$ is a $\mathds{Z}$-graded algebra with $W_j$ the homogeneous component of degree $j$. It is obvious that $\overline{W}_j=W_{-j}$.

\section{The structure of the centralizer}
For $P\in W$, we let $\Z(P)$ denote the centralizer of $P$. That is, the subalgebra of $W$ consisting of all the $Q$'s such that $PQ = QP$. The main purpose of this Section is to prove that if there exist $Q\in W$ such that $[Q,P] = 1$, then $\Z(P) = k[P]$.

\begin{lemma}\label{caracterizacion de Wn} The following facts hold:

\begin{enumerate}

\smallskip

\item For each $j\in \mathds{N}_0$ and $f\in k[Z]$,
$$
f(XY)X^j=X^jf(XY+j)\quad\text{and}\quad Y^jf(XY)=f(XY+j)Y^j.
$$

\smallskip

\item For each $j\in \mathds{Z}$,
$$
W_j=\begin{cases}
X^jk[XY] & \text{if $j\ge 0$,}\\
k[XY]Y^{-j} & \text{if $j< 0$.}
\end{cases}
$$
\end{enumerate}
\end{lemma}

\begin{proof} See~\cite[3.2, 3.3]{D}.
\end{proof}

\begin{theorem}\label{estructura del centralizador de un elemento homogeneo} Let $P\in W_r\setminus k$. If $r=0$, then $\Z(P)=k[XY]$. On the other hand, if $r\ne 0$, then $\dim_k\bigl(\Z(P)\cap W_j\bigr)\le 1$ for each $j\in \mathds{Z}$.
\end{theorem}

\begin{proof} First assume that $r=0$. It is evident that $k[XY]\subseteq \Z(P)$. In order to check the opposite inclusion, we note that $\Z(P)$ is a graded subalgebra of $W$. Let $Q\in \Z(P)\setminus\{0\}$ be an homogeneous element. By item~(3) of Proposition~\ref{conmutadores}, we have $P\sim Q$ and so $v(Q) = 0$. Now assume that $r> 0$. Write $P=X^rf(XY)$ and take
$$
Q\in \Z(P)\cap W_j\setminus \{0\}.
$$
By item~(3) of Proposition~\ref{conmutadores}, we know that $P\sim Q$. So $j>0$ and $Q=X^jg(XY)$ for some $g\in k[Z]$. A direct computation using that
$$
f(XY) X^j=X^j f(XY+j)\quad\text{ and }\quad g(XY) X^r=X^r g(XY+r)
$$
shows that $[P,Q]=0$ if and only if
$$
\frac{f(Z+j)}{f(Z)}=\frac{g(Z+r)}{g(Z)}.
$$
Hence, in order to finish the proof in the case $r>0$, it suffices to check that if $g_1\ne 0$ satisfies
$$
\frac{g_1(Z+r)}{g_1(Z)}=\frac{g(Z+r)}{g(Z)},
$$
then $g_1=\lambda g$ for some $\lambda\in k$. Let $n_0\in \mathds{Z}$ be such that $g(n)\ne 0$ and $g_1(n)\ne 0$ for all $n\ge n_0$. Set $\lambda:=g_1(n_0)/g(n_0)$. We have
$$
g_1(n_0+r)=g_1(n_0)g(n_0+r)/g(n_0)=\lambda g(n_0+r).
$$
Iterating the same argument we obtain $g_1(n_0+ir)=\lambda g(n_0+ir)$ for all $i>0$. So the polynomials $g_1$ and $\lambda g$ coincide. We left the case $r<0$ to the reader. Use the symmetric version of Proposition~\ref{conmutadores}.
\end{proof}

For $P\in W_+$, there exists exactly one pair $(i,j)\in \mathds{N}_0^2$, with $\gcd(i,j)=1$, such that $w(P) = (ri,rj)$ for an $r>0$. We define
$$
\Z_l(P) = \{Q\in \Z(P)\setminus\{0\}: w(Q)=(li,lj)\}
$$
for each $l\ge 0$. By item~(3) of Proposition~\ref{conmutadores},
\begin{equation}
\Z(P)\setminus\{0\} = \bigcup_{l\ge 0} \Z_l(P).\label{eq5}
\end{equation}
In a similar way we define $\ov{\Z}_l(P)$ for every $P\in \ov{W}_+$ and $l\ge 0$.

\begin{theorem}\label{centralizador} For all $P\in W_+$ with $w(P) = (ri,rj)$ as above, there exist elements $R_l\in \Z_l(P)$ such that:
\begin{enumerate}

\smallskip

\item $\Z(P) = \bigoplus_{l\in L} k R_l$, with $L=L(P)\subset \mathds{N}_0$,

\smallskip

\item $\ell_t(R_l)=X^{li}Y^{lj}$,

\smallskip

\item $\ell(R_l)\ell(R_h)=\ell(R_{l+h})$.
\end{enumerate}
For $P\in \overline{W}_+$, the symmetric result holds.
\end{theorem}

\begin{proof} Let $L(P) = \{l\ge 0: \Z_l(P)\ne \emptyset\}$. For all $l\in L(P)$, we fix an $R_l\in \Z_l(P)$ with $\ell_t(R_l)=X^{li}Y^{lj}$. Clearly $\bigoplus_{l\in L} k R_l\subseteq \Z(P)$. We now prove the opposite inclusion. Let $Q\in \Z(P)\setminus\{0\}$ with $w(Q)=(hi,hj)$. Since
$$
\ell(Q),\ell(R_h)\in \Z(\ell(P))\cap W_{h(i-j)},
$$
by Theorem~\ref{estructura del centralizador de un elemento homogeneo} there exists $\lambda_h \in k$ such that $\ell(Q) = \lambda_h\ell(R_h)$. If $Q = \lambda_h R_h$, then $Q\in \bigoplus_{l \in L} k R_l$. Otherwise, $Q-\lambda_h R_h\in \Z(P)$ and, by~\eqref{eq5},
$$
w(Q-\lambda_h R_h)=(h'i,h'j),
$$
with $0\le h'<h$. The same argument yields $\lambda_{h'}$ such that $Q = \lambda_h R_h + \lambda_{h'}R_{h'}$ or
$$
Q-\lambda_h R_h-\lambda_{h'}R_{h'}\in \Z(P)\quad\text{and}\quad w(Q-\lambda_h R_h - \lambda_{h'}R_{h'}) = (h''i,h''j),
$$
with $0\le h''<h'$. Iterating the argument we obtain that $Q \in \bigoplus_{l\in L} k R_l$. So items~(1) and~(2) hold. Finally, since
$$
\ell(R_l)\ell(R_h),\ell(R_{l+h})\in \Z(\ell(P))\cap W_{(l+h)(i-j)},
$$
again, by Theorem~\ref{estructura del centralizador de un elemento homogeneo}, there exists $\lambda\in k$ such that $\ell(R_l)\ell(R_h)=\lambda \ell(R_{l+h})$. By item~(2), necessarily $\lambda=1$.

When $P\in \ov{W}_+$ the same argument works.
\end{proof}

\begin{remark}\label{R_0=1} We assert that $R_0 = 1$. Otherwise $R_0-1\in \Z(P)\setminus\{0\}$, and so $R_0-1\sim P$. But this is impossible since $v(P)>0$ and $v(R_0-1)<0$.
\end{remark}

\begin{remark} Let $P\in W\setminus\{0\}$. If $\overline v(P)\le 0$ and $v(P)\le 0$, then, necessarily, $\overline v(P)=0=v(P)$ and $P\in W_0$. In this case $\Z(P)=k[XY]$ if $P\notin k$ and $\Z(P)=W$ if $P\in k$.
\end{remark}

\begin{remark}\label{caso homogeneo} Note that in general $R_lR_h \ne R_{l+h}$. So item~(1) does not yield a graduation on $\Z(P)$. However if $P\in W_+$ is homogeneous (that is $P = \ell(P)$), then we can assume that $R_l = \ell(R_l)$. Hence, by item~(3) of the above theorem, $\Z(P)$ is graded and, therefore, the map $\varphi:\Z(P) \to k[Z]$, given by $\varphi(R_l) = Z^l$, is a monomorphism of graded algebras. Consequently $\Z(P)$ is a monomial algebra. A similar result holds for $P\in \overline{W}_+$, homogeneous.
\end{remark}

\begin{lemma}\label{submonoide} Let $L$ be an additive submonoid of $\mathds{N}_0$. Let $r_0 = \min(L\setminus \{0\})$ and $d = \gcd(L)$. For $0\le r <r_0$, let $L_r := \{l\in L:l\equiv r \pmod{r_0}\}$. Then
$$
L_r \ne \emptyset \Leftrightarrow d\mid r\qquad\text{and}\qquad L = \bigcup_{r=0}^{r_0-1} L_r.
$$
\end{lemma}

\begin{proof} First note that the second assertion is trivial and that $L_r\ne \emptyset$ clearly implies that $d\mid r$. We now prove the opposite implication. Let $n_0 := r_0/d$.
Since the group generated by $L$ is $d\mathds{Z}$, there exist $l_0,l_1\in L$ such that $d = l_1-l_0$. So
$$
l_0n_0,l_0(n_0-1)+l_1,l_0(n_0-2)+2l_1,\dots,l_0+(n_0-1)l_1
$$
are elements in the $n_0$ different $L_r$ with $d\mid r$, as desired.
\end{proof}

For a fixed $P\in W_+$ we consider $d=\gcd\{l:l\in L(P)\}$ and set $\deg Q = l/d$ for $Q\in \Z_l(P)$. Note that

\begin{itemize}

\smallskip

\item[-] $\deg(Q_1)>\deg(Q_2)$ if and only if $v(Q_1)>v(Q_2)$,

\smallskip

\item[-] for a polynomial $T\in k[S]\setminus k$, where $S\in \Z(P)$,
$$
\deg T' = \deg T-\deg S,
$$
in which $T'$ denotes the usual derivative of the polynomial $T$.

\smallskip

\end{itemize}
In the sequel we will use these facts again and again without explicit mention.

\smallskip

For $P\in W_+$, choose an element $S_0$ of minimal degree in $\Z(P)\setminus k$ and set $n_0:=\deg S_0 $. For each $0<l<n_0$, set
$$
\widehat \Z_l(P):=\{ R\in \Z(P)\setminus\{0\}:\deg(R)\equiv l \pmod{n_0}\}.
$$
Since $L(P)$ is an additive submonoid of $\mathds{N}_0$, from Lemma~\ref{submonoide} it follows that the $\widehat \Z_l(P)$'s are not empty. Fix $S_l\in \widehat \Z_l(P)$ of minimal degree.

\begin{corollary}\label{Z(P) como suma directa} We have:
\begin{equation}\label{eq1}
\Z(P)=k[S_0]\oplus k[S_0]S_1\oplus \cdots \oplus k[S_0]S_{n_0-1}.
\end{equation}
For $P\in \overline{W}_+$ the symmetric result holds.
\end{corollary}

\begin{proof} It is clear that
$$
k[S_0]\oplus k[S_0]S_1\oplus \cdots \oplus k[S_0]S_{n_0-1}\subseteq \Z(P).
$$
Consider $R\in \Z(P)\setminus\{0\}$. We will prove by induction on $\deg R $ that $R$ is contained in the right side of the equality~\eqref{eq1}. If $\deg R =0$, then by item~(1) of Theorem~\ref{centralizador} and Remark~\ref{R_0=1}, we have $R\in k\subseteq k[S_0]$. Otherwise, there exist $0\le r<n_0$ and $l\in \mathds{N}_0$ such that
$$
\deg(S_0^l S_r) = \deg S_r + l n_0=\deg R.
$$
Hence, by Theorem \ref{estructura del centralizador de un elemento homogeneo}, we can find
$\lambda\in k$ such that $\lambda \ell(S_0^l S_r)=\ell(R)$. Then, by inductive hypothesis, $R-\lambda S_0^l S_r$ belongs to the right side of~\eqref{eq1}, since
$$
R-\lambda S_0^l S_r = 0\quad\text{or}\quad \deg(R-\lambda S_0^l S_r)<\deg R.
$$
This finishes the proof.
\end{proof}

The $S_i$'s in the previous corollary can be chosen monic (i.e., with $\ell_c(S_i)=1$), and we will do it so from now on.

\smallskip

By~\cite[Corollary 4.5]{D} we know that $\Z(P)$ is a commutative algebra, and so
$$
\partial(T(S)) = T'(S)\partial(S)
$$
for any derivation $\partial:\Z(P)\to \Z(P)$ and any $T\in k[S]$ with $S\in \Z(P)$.

\begin{lemma}\label{grado de la derivada de S_i} Let $P\in W_+$ and let
$$
\Z(P)=k\oplus k[S_0]S_0\oplus k[S_0]S_1\oplus \cdots \oplus k[S_0]S_{n_0-1}
$$
be as in Corollary~\ref{Z(P) como suma directa}. Let
$$
\partial:\Z(P)\to \Z(P)
$$
be a derivation and set $J:=\{r: \partial (S_r)\ne 0\}$. Then
$$
\deg S_r-\deg\partial (S_r)=\deg S_t -\deg\partial (S_t)
$$
for all $r,t\in J$.
\end{lemma}

\begin{proof} Set $g_r:=\deg S_r $ and $w_r:=\deg \partial (S_r)$. Note that $g_0=n_0$. Consider the set
$$
D := \{w_r-g_r: r\in J\}.
$$
We must prove that $\# D=1$. Assuming that this is not the case will lead us to a contradiction. Take $r,t$ such that
$$
w_t-g_t=\max D>w_r-g_r.
$$
Since $\deg(S_t^{g_r})=\deg(S_r^{g_t})$, by Theorem~\ref{estructura del centralizador de un elemento homogeneo} we know that there exist $\lambda_0,\lambda_1\in k$ and $P_i\in k[S_0]$ such that
$$
S_t^{g_r}=\lambda_1 S_r^{g_t}+\lambda_0+\sum_{i=0}^{n_0-1}P_i S_i\quad\text{with $P_iS_i= 0$ or $\deg(P_i S_i)<g_t g_r$.}
$$
This implies that
\begin{equation}\label{eq2}
U:=\deg\partial (S_t^{g_r}) \le \max\{ \deg\partial (S_r^{g_t}),\deg(P_i \partial (S_i)), \deg(\partial (P_i) S_i)\},
\end{equation}
where we only consider non zero terms $P_i \partial (S_i)$ and $\partial (P_i) S_i$ (note that $\partial(P_i)\ne 0$ only if $\partial(S_0)\ne 0$ and $P_i\not\in k$). But~\eqref{eq2} is impossible since $U$ is strictly greater than each of the terms on the right side. In order to check this, note that
\begin{align*}
\partial (S_t^{g_r})&=g_rS_t^{g_r-1}\partial (S_t),\\[3pt]
\partial (S_r^{g_t})&=g_tS_r^{g_t-1}\partial (S_r),\\[3pt]
\partial (P_i) S_i&=P_i'\partial (S_0)S_i,
\end{align*}
and hence,
\begin{align*}
&\deg\partial (S_t^{g_r}) =g_r g_t + w_t-g_t=U,\\[3pt]
&\deg\partial (S_r^{g_t}) =g_r g_t + w_r-g_r<U,\\[3pt]
&\deg(P_i\partial(S_i))= \deg P_i +w_i=\deg(P_i S_i)+w_i-g_i<g_t g_r +w_i-g_i\le U,\\[3pt]
&\deg(\partial (P_i) S_i)= \deg P_i -g_0+w_0+g_i=\deg(P_i S_i)+w_0-g_0< U.
\end{align*}
This concludes the proof.
\end{proof}

\begin{proposition}\label{grado de R} Let $P\in W_+$. If $\partial\colon Z(P)\to Z(P)$ is a non zero derivation, then $\ker \partial = k$. Moreover, using the same notation as in the previous lemma, if
$$
R = \lambda_0+\sum_{i=0}^{n_0-1} R_iS_i\, \in \, \Z(P)\setminus k,
$$
then
$$
\deg\partial (R) = \deg R + \deg \partial(S_j)-\deg S_j\quad\text{for all $j$.}
$$
\end{proposition}

\begin{proof} We first prove that
\begin{equation}
\deg\partial (R) \le \deg R + \deg \partial(S_j)-\deg S_j\quad\text{for all $j\in J$.} \label{eq6}
\end{equation}
By hypothesis $J\ne \emptyset$. As in the proof of Lemma~\ref{grado de la derivada de S_i}, we set $g_j := \deg S_j$ and $w_j :=\deg\partial (S_j)$ for all $j\in J$. Now fix $j\in J$. Since
\begin{align*}
&\deg(\partial(R))\le \max\{\deg\partial(R_iS_i):\partial(R_iS_i)\ne 0\}
\intertext{and}
&\deg(R) = \max\{\deg(R_iS_i):R_iS_i\ne 0\},
\end{align*}
it suffices to show that if $\partial (R_iS_i)\ne 0$, then
$$
\deg\partial (R_iS_i) \le \deg (R_iS_i) + w_j - g_j.
$$
By Lemma~\ref{grado de la derivada de S_i}, if $i\in J$, then
\begin{equation}
\deg(R_i\partial(S_i)) = \deg(R_iS_i) + w_i - g_i = \deg(R_iS_i) + w_j - g_j\label{eq7}
\end{equation}
and if $0\in J$ and $R_i\notin k$, then
\begin{equation}
\begin{aligned}
\deg (R_i'\partial(S_0)S_i)&=\deg R_i-g_0+w_0+g_i\\
&=\deg(R_i S_i)+w_0-g_0\\
&=\deg (R_iS_i)+w_j-g_j.
\end{aligned}\label{eq8}
\end{equation}
So,
$$
\deg\partial(R_iS_i) = \deg\bigl(R_i\partial(S_i)+R_i'\partial(S_0)S_i\bigr)\le \deg(R_iS_i) + w_j-g_j,
$$
as desired. We now prove that $\ker \partial = k$ or, equivalently, that $J=\{0,\dots,n_0-1\}$. Let $0\le h < n_0$. As in the proof of Lemma~\ref{grado de la derivada de S_i} we can write
\begin{equation}
S_h^{g_j} = \lambda_1 S_j^{g_h}+\lambda_0+\sum_{i=0}^{n_0-1} P_iS_i\quad\text{with $P_iS_i=0$ or $\deg(P_i S_i)<g_jg_h$.}\label{eq10}
\end{equation}
Note that $\ell_c(S_h)=\ell_c(S_j)=1$ and so $\lambda_1=1$. Hence, as we have seen above, if $\partial(P_iS_i) \ne 0$, then
\begin{equation}
\deg\partial (P_iS_i) \le \deg (P_iS_i) + w_j - g_j < g_jg_h+w_j-g_j = \deg \partial(S_j^{g_h}).\label{eq11}
\end{equation}
Consequently, since $j\in J$,
\begin{equation}
g_jS_h^{g_j-1}\partial(S_h) = \partial(S_h^{g_j}) = \partial(S_j^{g_h})+ \sum_{i=0}^{n_0-1} \partial(P_iS_i)\ne 0,\label{eq12}
\end{equation}
which implies that $h\in J$. It remains to prove that the equality holds in~\eqref{eq6}. For this it will be sufficient to check that for our fixed $j$ and any $i$,
\begin{equation}
\deg\partial (R_iS_i) = \deg (R_iS_i) + w_j - g_j.\label{eq9}
\end{equation}
In fact, if this is true and
$$
\deg\partial(R_hS_h) = \max\{\deg\partial(R_iS_i):\partial(R_iS_i)\ne 0\},
$$
then also
$$
\deg(R_hS_h) = \max\{\deg (R_iS_i):R_iS_i\ne 0\},
$$
and so
$$
\deg\partial(R) = \deg\partial(R_hS_h) = \deg (R_hS_h) + w_j - g_j = \deg(R) + w_j - g_j.
$$
Now we are going to prove~\eqref{eq9}. By~\eqref{eq7} and~\eqref{eq8} we are reduced to show that
$$
M:=\ell(R_i')\ell(\partial(S_0))\ell(S_i) + \ell(R_i)\ell(\partial(S_i))\ne 0\qquad\text{for all $R_i\in Z(P)\setminus k$}.
$$
Write
$$
R_i = a_n S_0^n + \cdots + a_0\quad\text{with $a_n\ne 0$.}
$$
By~\eqref{eq10}, we know that $\ell(S_i^{ g_0}) = \ell(S_0^{g_i})$ and by~\eqref{eq11} and~\eqref{eq12}, that
$$
g_0\ell(S_i^{g_0-1})\ell(\partial(S_i)) = \ell(\partial(S_i^{ g_0})) = \ell(\partial (S_0^{g_i})) = g_i\ell(S_0^{g_i-1}) \ell(\partial( S_0)).
$$
Hence, in the quotient field of $\Z(P)$,
\begin{align*}
\ell(\partial(S_i))&=\frac{g_i}{g_0}\ell(\partial(S_0))\frac{\ell(S_0)^{g_i-1}\ell(S_i)}
{\ell(S_i)^{g_0-1}\ell(S_i)}\\[3pt]
&=\frac{g_i}{g_0}\ell(\partial(S_0))\frac{\ell(S_0)^{g_i-1}\ell(S_i)}
{\ell(S_0)^{g_i}}\\[3pt]
&=\frac{g_i}{g_0}\ell(\partial(S_0))\frac{\ell(S_i)}
{\ell(S_0)},
\end{align*}
which implies that
$$
\ell(\partial(S_i))\ell(S_0)=\frac{g_i}{g_0}\ell(\partial(S_0))\ell(S_i)
$$
in $Z(P)$. Therefore, since $S_0\in W_+$ and $n>0$,
\begin{align*}
M&=\ell(R_i')\ell(\partial(S_0))\ell(S_i) + \ell(R_i)\ell(\partial(S_i))\\
&= n a_n\ell(S_0)^{n-1}\ell(\partial(S_0))\ell(S_i)+a_n \ell(S_0)^{n-1}\ell(S_0)\ell(\partial(S_i))\\
&= a_n\ell(S_0)^{n-1}\ell(\partial(S_0))\ell(S_i)\biggl(n+\frac{g_i}{g_0}\biggr)\ne 0,
\end{align*}
as desired.
\end{proof}

Let $P,Q\in W$ such that $[Q,P]=1$. Since
$$
[P,[Q,R]]=[[P,Q],R]+[Q,[P,R]]=[Q,[P,R]],
$$
$ad_Q:=[Q,-]$ defines a derivation $\partial$ from $\Z(P)$ to $\Z(P)$.

\begin{theorem} If $P,Q\in W$ satisfy $[Q,P]=1$, then $\Z(P) = k[P]$.
\end{theorem}

\begin{proof} Using Lemma~\ref{caracterizacion de Wn} it is easy to check that if $P\in W_0\cap\ov{W}_0 = k[XY]$, then there is no $Q$ such that $[Q,P]=1$. So $P\in W_+\cup \ov{W}_+$. Assume $P\in W_+$. Let $\partial := ad_Q$ and let $S_i,g_i,w_i$ be as above. Since $g_i > g_0$ for $i>0$, from Lemma~\ref{grado de la derivada de S_i} it follows that $w_i>w_0$ for $i>0$. Write
$$
P=\lambda_0+\sum_{i=0}^{n_0-1} P_iS_i.
$$
By Proposition~\ref{grado de R},
$$
\deg \partial(P_iS_i)=\deg(P_iS_i)+w_i-g_i=w_i+\deg P_i\equiv w_i\pmod{n_0}
$$
for each $P_i\ne 0$. Moreover, by  Lemma~\ref{grado de la derivada de S_i}
$$
w_i\not\equiv w_j\pmod{n_0}\quad\text{for $i\ne j$,}
$$
and so
$$
0=\deg \partial(P)=\max \{w_i+\deg P_i:P_i\ne 0\}.
$$
Hence $P_i=0$ for $i>0$ since $w_i>w_0\ge 0$, and $w_0=0=\deg P_0$. Consequently $P=\lambda_0+\lambda_1 S_0$. We claim that $n_0=1$, which concludes the proof. In fact, if $n_0>1$, then $\partial(S_l)\in \Z(P)$ for $0<l<n_0$, and so, by Lemma~\ref{grado de la derivada de S_i},
$$
\deg \partial(S_l)=w_l=g_l-g_0\equiv g_l-n_0\equiv l \pmod{n_0}.
$$
Therefore
$$
\partial(S_l)\in \widehat\Z_l(P)\quad\text{and}\quad\deg \partial(S_l)<\deg S_l,
$$
which contradicts the minimality of $\deg S_l$. For $P\in \overline{W}_+$, the same argument works, using the symmetric versions of Lemma~\ref{grado de la derivada de S_i} and Proposition~\ref{grado de R}.
\end{proof}

\end{document}